\documentclass[11pt]{amsart}
\usepackage{bbm,graphicx}

\usepackage[margin=1.25in]{geometry}

\newtheorem{thm}{Theorem}
\newtheorem{cor}[thm]{Corollary}
\newtheorem{lemma}[thm]{Lemma}
\newtheorem{prop}[thm]{Proposition}

\newcommand{\R}{\mathbb{R}}
\newcommand{\E}{\mathbb{E}}
\newcommand{\Prob}{\mathbb{P}}
\newcommand{\N}{\mathbb{N}}

\newcommand{\abs}[1]{\left\vert #1 \right\vert}
\newcommand{\norm}[1]{\left\Vert #1 \right\Vert}
\newcommand{\eps}{\varepsilon}
\newcommand{\dfn}[1]{\textbf{#1}}

\newcommand{\ind}[1]{\mathbbm{1}_{#1}}
\DeclareMathOperator{\Var}{Var}

\allowdisplaybreaks[3]
\author{Elizabeth S.\ Meckes}

\author{Mark W.\ Meckes}

\address{Department of Mathematics, Case Western Reserve University,
10900 Euclid Ave., Cleveland, Ohio 44106, U.S.A.}

\email{elizabeth.meckes@case.edu}

\address{Department of Mathematics, Case Western Reserve University,
10900 Euclid Ave., Cleveland, Ohio 44106, U.S.A.}

\email{mark.meckes@case.edu}

\title{On the equivalence of modes of convergence for log-concave
  measures}

\begin{document}

\begin{abstract}
  An important theme in recent work in asymptotic geometric analysis
  is that many classical implications between different types of
  geometric or functional inequalities can be reversed in the presence
  of convexity assumptions.  In this note, we explore the extent to
  which different notions of distance between probability measures are
  comparable for log-concave distributions. Our results imply that
  weak convergence of isotropic log-concave distributions is
  equivalent to convergence in total variation, and is further
  equivalent to convergence in relative entropy when the limit measure
  is Gaussian.
\end{abstract}

\maketitle


\section{Introduction and statements of results}

An important theme in recent work in asymptotic geometric analysis is
that many classical implications between different types of geometric
or functional inequalities can be reversed in the presence of
convexity.  A particularly striking recent example is the work of E.\
Milman \cite{Milman1,Milman2,Milman3}, showing for example that, on a
Riemannian manifold equipped with a probability measure satisfying a
convexity assumption, the existence of a Cheeger inequality, a
Poincar\'e inequality, and exponential concentration of Lipschitz
functions are all equivalent.  Important earlier examples of this
theme are C.\ Borell's 1974 proof of reverse H\"older inequalities for
log-concave measures \cite{Borell}, and K.\ Ball's 1991 proof of a
reverse isoperimetric inequality for convex bodies \cite{Ball}.

In this note, we explore the extent to which different notions of
distance between probability measures are comparable in the presence
of a convexity assumption.  Specifically, we consider
\dfn{log-concave} probability measures; that is, Borel probability
measures $\mu$ on $\R^n$ such that for all nonempty compact sets $A, B
\subseteq \R^n$ and every $\lambda \in (0, 1)$,
\[
\mu(\lambda A + (1-\lambda) B) \ge \mu(A)^\lambda \mu(B)^{1-\lambda}.
\]
We moreover consider only those log-concave probability measures $\mu$
on $\R^n$ which are \dfn{isotropic}, meaning that if $X \sim \mu$
then
\[
\E X=0
\qquad \text{and} \qquad \E XX^T=I_n.
\]

The following distances between probability measures $\mu$ and $\nu$
on $\R^n$ appear below.
\begin{enumerate}
\item The \dfn{total variation distance} is defined by
  \begin{equation*}
    d_{TV}(\mu,\nu):=2\sup_{A\subseteq\R^n}\abs{\mu(A)-\nu(A)},
  \end{equation*}
  where the supremum is over Borel measurable sets.  
\item The \dfn{bounded Lipschitz distance} is defined by
  \begin{equation*}
    d_{BL}(\mu,\nu):=\sup_{\norm{g}_{BL}\le 1}\abs{\int g \ d\mu-\int
      g \ d\nu},
  \end{equation*}
  where the bounded-Lipschitz norm $\norm{g}_{BL}$ of $g:\R^n\to\R$ is
  defined by
  \[
  \norm{g}_{BL} := \max\left\{\norm{g}_\infty,\ \sup_{x\neq
      y}\frac{\abs{g(x)-g(y)}}{\norm{x-y}}\right\}
  \] 
  and $\norm{\cdot}$ denotes the standard Euclidean norm on $\R^n$.
  The bounded-Lipschitz distance is a metric for the weak topology on
  probability measures (see, e.g., \cite[Theorem 11.3.3]{Dudley}).
\item The \dfn{$L_p$ Wasserstein distance} for $p\ge 1$ is defined by
  \begin{equation*}
    W_p(\mu,\nu) :=
    \inf_\pi\left[\int\norm{x-y}^p \ d\pi(x,y)\right]^{\frac{1}{p}},
  \end{equation*}
  where the infimum is over couplings $\pi$ of $\mu$ and $\nu$; that
  is, probability measures $\pi$ on $\R^{2n}$ such that
  $\pi(A\times\R^n)=\mu(A)$ and $\pi(\R^n\times B)=\nu(B)$. The $L_p$
  Wasserstein distance is a metric for the topology of weak
  convergence plus convergence of moments of order $p$ or less. (See
  \cite[Section 6]{Villani2} for a proof of this fact, and a lengthy
  discussion of the many fine mathematicians after whom this distance
  could reasonably be named.)

\item If $\mu$ is absolutely continuous with respect to $\nu$, the
  \dfn{relative entropy}, or \dfn{Kullback--Leibler divergence}
  is defined by
  \begin{equation*}
    H(\mu\mid\nu) := \int
    \left(\frac{d\mu}{d\nu}\right)\log\left(\frac{d\mu}{d\nu}\right) \
    d\nu
    = \int \log\left(\frac{d\mu}{d\nu}\right) \ d\mu.
  \end{equation*}
\end{enumerate}

\bigskip

It is a classical fact that for any probability measures $\mu$ and
$\nu$ on $\R^n$,
\begin{equation}\label{E:bltv-classical}
  d_{BL}(\mu,\nu)\le d_{TV}(\mu,\nu).
\end{equation}
This follows from a dual formulation of total variation distance: the
Riesz representation theorem implies that
\begin{equation} \label{E:tv-dual}
d_{TV}(\mu,\nu) = \sup \left\{ \abs{\int g \ d\mu-\int g \ d\nu}
    \colon g \in C(\R^n), \ \norm{g}_\infty \le 1 \right\}.
\end{equation}
In the case that $\mu$ and $\nu$ are log-concave, there is the
following complementary inequality.

\begin{prop} 
  \label{T:tv-bl}
  Let $\mu$ and $\nu$ be log-concave isotropic probability measures on
  $\R^n$. Then 
  \[
  d_{TV}(\mu, \nu) \le C \sqrt{ n d_{BL}(\mu, \nu)}.
  \]
\end{prop}

In this result and below, $C, c$, etc.\ denote positive constants
which are independent of $n$, $\mu$, and $\nu$, and whose values may
change from one appearance to the next.

\medskip

In the special case in which $n=1$ and $\nu = \gamma_1$, Brehm, Hinow,
Vogt and Voigt proved a similar comparison between total variation
distance and Kolmogorov distance $d_K$.

\begin{prop}[{\cite[Theorem 3.3]{BHVV}}]
  Let $\mu$ be a log-concave measure on $\R$. Then
  \[
  d_{TV}(\mu, \gamma_1) \le C \sqrt{\max \left\{ 1, \log (1/d_K(\mu,
      \gamma_1)) \right\} d_K(\mu, \gamma_1)}.
  \]
\end{prop}

 Together with \eqref{E:bltv-classical}, Proposition \ref{T:tv-bl}
implies the following.

\begin{cor} 
  \label{T:weak-tv}
  On the family of isotropic log-concave probability measures on
  $\R^n$, the topologies of weak convergence and of total variation
  coincide.
\end{cor}

Corollary \ref{T:weak-tv} will probably be unsurprising to experts,
but we have not seen it stated in the literature.

Proposition \ref{T:tv-bl} and Corollary \ref{T:weak-tv} are false
without the assumption of isotropicity.  For example, a sequence of
nondegenerate Gaussian measures $\{ \mu_k \}_{k \in \N}$ on $\R^n$ may
weakly approach a Gaussian measure $\mu$ supported on a
lower-dimensional subspace, but $d_{TV}(\mu_k, \mu) = 2$ for every
$k$.  It may be possible to extend Corollary \ref{T:weak-tv} to a
class of log-concave probability measures with, say, a nontrivial
uniform lower bound on the smallest eigenvalue of the covariance
matrix, but we will not pursue this here.

\medskip

The Kantorovitch duality theorem (see \cite[Theorem 5.10]{Villani2})
gives a dual formulation of the $L_1$ Wasserstein distance similar to
the formulation of total variation distance in \eqref{E:tv-dual}:
\begin{equation*}
  W_1(\mu,\nu)=\sup_g\abs{\int g \ d\mu-\int g \ d\nu},
\end{equation*}
where the supremum is over 1-Lipschitz functions $g:\R^n\to\R$.  An
immediate consequence is that for any probability measures $\mu$ and
$\nu$,
\begin{equation*}
  d_{BL}(\mu,\nu)\le W_1(\mu,\nu).
\end{equation*}
The following complementary inequality holds in the log-concave case.

\begin{prop} 
  \label{T:w1-bl}
  Let $\mu$ and $\nu$ be log-concave isotropic probability
  measures on $\R^n$. Then
  \begin{equation} \label{E:w1-bl}
  W_1(\mu, \nu) \le C \max \left\{ \sqrt{n}, \log 
      \left(\frac{\sqrt{n}}{d_{BL}(\mu,\nu)} \right) \right\}
    d_{BL}(\mu, \nu).
  \end{equation}
\end{prop}

The following graph of $f(x) = \max
\left\{1,\log\left(\frac{1}{x}\right) \right\} x$ may be helpful in
visualizing the bounds in Proposition \ref{T:w1-bl} and the results
below.

\begin{center}\includegraphics[width=2in]{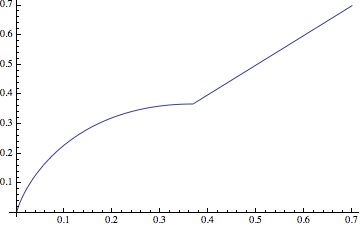}\end{center}
In particular, when $d_{BL}$ is moderate, we simply have $W_1 \le
C\sqrt{n} d_{BL}$. When $d_{BL}$ is small, the right hand side
of \eqref{E:w1-bl} is not quite linear in $d_{BL}$, but is
$o\bigl(n^{\eps/2} d_{BL}^{1-\eps}\bigr)$ for each $\eps > 0$.

From H\"older's inequality, it is immediate that if $p\le q$, then
$W_p(\mu,\nu)\le W_q(\mu,\nu)$.  In the log-concave case, we have the
following.

\begin{prop} 
  \label{T:wq-wp}
  Let $\mu$ and $\nu$ be isotropic log-concave probability measures on
  $\R^n$ and let $1 \le p < q$. Then
  \[
  W_q(\mu,\nu)^q \le C \left(\max \left\{ \sqrt{n},
      \log \left( \frac{\left(c \max\{q, \sqrt{n}\}\right)^q }{W_p(\mu, \nu)^{p}}
      \right) \right\} \right)^{q-p} W_p(\mu, \nu)^{p}.
  \]
\end{prop}

Because the bounded-Lipschitz distance metrizes the weak topology, and
convergence in $L_p$ Wasserstein distance implies convergence of
moments of order smaller than $p$, Propositions \ref{T:w1-bl} and
\ref{T:wq-wp} imply the following.

\begin{cor}
  \label{T:weak-moments}
  Let $\mu$, $\{\mu_k\}_{k \in \N}$ be isotropic log-concave
  probability measures on $\R^n$ such that $\mu_k \to \mu$ weakly.
  Then all moments of the $\mu_k$ converge to the corresponding
  moments of $\mu$.
\end{cor}

The following, known as the Csisz\'ar--Kullback--Pinsker inequality,
holds for any probability measures $\mu$ and $\nu$:
\begin{equation}\label{E:tvh-classical} 
  d_{TV}(\mu,\nu) \le \sqrt{2 H(\mu \mid \nu)}.
\end{equation}
(See \cite{BoVi} for a proof, generalizations, and original
references.)  Unlike the other notions of distance considered above,
$H(\cdot\mid\cdot)$ is not a metric, and $H(\mu\mid\nu)$ can only be
finite if $\mu$ is absolutely continuous with respect to $\nu$.
Nevertheless, it is frequently used to quantify convergence;
\eqref{E:tvh-classical} shows that convergence in relative entropy is
stronger than convergence in total variation.  Convergence in relative
entropy is particularly useful in quantifying convergence to the
Gaussian distribution, and it is in that setting that
\eqref{E:tvh-classical} can be essentially reversed under an
assumption of log-concavity.

\begin{prop} 
  \label{T:h-tv} 
  Let $\mu$ be an isotropic log-concave probability measure on $\R^n$,
  and let $\gamma_n$ denote the standard Gaussian distribution on
  $\R^n$. Then
  \[
  H(\mu \mid \gamma_n) \le C \max \left\{\log^2
    \left(\frac{n}{d_{TV}(\mu, \gamma_n)}\right), n \log (n + 1) \right\}
    d_{TV}(\mu, \gamma_n).
  \]
\end{prop}

The proof of Proposition \ref{T:h-tv} uses a rough bound on the
isotropic constant $L_f = \norm{f}_\infty^{1/n}$ of the density $f$ of
$\mu$.  Better estimates are available but only result in a change in
the absolute constants in our bound.  In the case that the isotropic
constant is bounded independent of $n$ (e.g.\ if $\mu$ is the uniform
measure on an unconditional convex body, or if the hyperplane
conjecture is proved), then the bound above can be improved slightly
to
  \[
  H(\mu \mid \gamma_n) \le C \max \left\{\log^2
    \left(\frac{n}{d_{TV}(\mu, \gamma_n)}\right), n \right\}
  d_{TV}(\mu, \gamma_n).
  \]

\begin{cor} 
  \label{T:weak-tv-h}
  Let $\{\mu_k\}_{k\in \N}$ be isotropic log-concave probability
  measures on $\R^n$.  The following are equivalent:
  \begin{enumerate}
  \item $\mu_k \to \gamma_n$ weakly.
  \item $\mu_k \to \gamma_n$ in total variation.
  \item $H(\mu_k \mid \gamma_n) \to 0$.
  \end{enumerate}
\end{cor}

It is worth noting that Proposition \ref{T:h-tv} implies that B.\
Klartag's central limit theorem for convex bodies (proved in
\cite{Klartag1,Klartag2} in total variation) also holds in the \emph{a
  priori} stronger sense of entropy, with a polynomial rate of
convergence.

\section{Proofs of the results}

The proof of Proposition \ref{T:tv-bl} uses the following
deconvolution result of R.\ Eldan and B.\ Klartag.

\begin{lemma}[{\cite[Proposition 10]{ElKl}}]
  \label{T:ElKl}
  Suppose that $f$ is the density of an isotropic log-concave
  probability measure on $\R^n$, and for $t > 0$ define
  \[
  \varphi_t(x) = \frac{1}{(2\pi t^2)^{n/2}} e^{-\norm{x}^2 / 2t^2}.
  \]
  Then
  \[
  \norm{f - f * \varphi_t }_1 \le c n t.
  \]
\end{lemma}

\begin{proof}[Proof of Proposition \ref{T:tv-bl}]
  Let $g \in C(\R^n)$ with $\norm{g}_\infty \le 1$.  For $t > 0$, let
  $g_t = g * \varphi_t$, where $\varphi_t$ is as in Lemma
  \ref{T:ElKl}. It follows from Young's inequality that
  $\norm{g_t}_\infty \le 1$ and that $g_t$ is $1/t$-Lipschitz.  We
  have
  \begin{align*}
    \abs{ \int g \ d\mu - \int g \ d\nu} & 
    \le \abs{ \int (g-g_t) \ d\mu }
    + \abs{ \int g_t \ d\mu - \int g_t \ d \nu}
    + \abs{ \int (g_t-g) \ d\nu }.
  \end{align*}
  It is a classical fact due to C.\ Borell \cite{Borell} that a
  log-concave probability measures which is not supported on a proper
  affine subspace of $\R^n$ has a density.  If $f$ is the density of
  $\mu$, then by Lemma \ref{T:ElKl},
  \[
  \abs{\int (g - g_t) \ d\mu} = \abs{\int g (f - f*\varphi_t)}
  \le \norm{f - f*\varphi_t}_1 \le cnt,
  \]
  and
  \[
  \abs{\int (g - g_t) \ d\nu} \le cnt
  \]
  similarly.  Furthermore,
  \[
  \abs{\int g_t \ d\mu - \int g_t \ d\nu} \le d_{BL}(\mu, \nu)
  \norm{g_t}_{BL} \le d_{BL}(\mu, \nu) \max\{1, 1/t\}.
  \]
  Combining the above estimates and taking the supremum over $g$
  yields
  \[
  d_{TV}(\mu, \nu) \le d_{BL}(\mu, \nu) \max \{1, 1/t\} + cnt
  \]
  for every $t > 0$.  The proposition follows by picking $t =
  \sqrt{d_{BL}(\mu, \nu) / 2n} \le 1$.
\end{proof}

The remaining propositions all depend in part on the following deep
concentration result due to G.\ Paouris.

\begin{prop}[\cite{Paouris}]
  \label{T:Paouris}
  Let $X$ be an isotropic log-concave random vector in $\R^n$. Then
  \[
  \Prob \left[ \norm{X} \ge R \right] \le e^{-cR}
  \]
  for every $R \ge C \sqrt{n}$, and
  \[
  \left(\E \norm{X}^{p}\right)^{1/p} \le C \max\{\sqrt{n}, p\}
  \]
  for every $p \ge 1$.
\end{prop}

The following simple optimization lemma will also be used in the
remaining proofs.

\begin{lemma} 
  \label{T:min}
  Given $A, B, M, k > 0$,
  \[
  \inf_{t \ge M} \left( A t^k + B e^{-t}\right)
  \le A \left( 1 + \left(\max \left\{ M, \log(B/A) \right\}\right)^k \right).
  \]
\end{lemma}

\begin{proof}
  Set $t = \max \{ M, \log (B/A) \}$.
\end{proof}

\begin{proof}[Proof of Proposition \ref{T:w1-bl}]
  Let $g:\R^n \to \R$ be $1$-Lipschitz and without loss of generality
  assume that $g(0) = 0$, so that $\abs{g(x)} \le \norm{x}$. For
  $R > 0$ define
  \[
  g_R(x) = \begin{cases} -R & \text{if } g(x) < -R, \\
    g(x) & \text{if } -R \le g(x) \le R, \\
    R & \text{if } g(x) > R, \end{cases}
  \]
  and observe that $\norm{g_R}_{BL} \le \max\{1,R\}$.  Let $X \sim \mu$
  and $Y \sim \nu$. Then
  \begin{align*}
    \abs{\E g(X) - \E g(Y)}
    & \le \E \abs{g_R(X) - g_R(Y)}     
    + \E \abs{g(X) - g_R(X)} + \E \abs{g(Y) - g_R(Y)} \\
    & \le \max\{1, R\} d_{BL}(\mu,\nu) + \E \norm{X} \ind{\norm{X} \ge R}
    + \E \norm{Y} \ind{\norm{Y} \ge R}.
    \end{align*}
    By the Cauchy--Schwarz inequality and Proposition \ref{T:Paouris},
    \[
    \E \norm{X} \ind{\norm{X} \ge R}
    \le \sqrt{n \Prob\left[\norm{X} \ge R\right]}
    \le \sqrt{n} e^{-cR}
    \]
    for $R \ge C\sqrt{n}$, and the last term is bounded similarly.
    Combining the above estimates and taking the supremum over $g$
    yields
    \[
    W_1(\mu, \nu) \le \max\{1, R\} d_{BL}(\mu, \nu) + 2\sqrt{n} e^{-cR}
    \]
    for every $R \ge C \sqrt{n}$.  The proposition follows using Lemma
    \ref{T:min}.
\end{proof}

\begin{proof}[Proof of Proposition \ref{T:wq-wp}]
  Let $(X,Y)$ be a coupling of $\mu$ and $\nu$ on $\R^n \times \R^n$. Then for each $R > 0$,
  \begin{align*}
    \E \norm{X-Y}^q 
    & \le R^{q-p}\E \left[\norm{X - Y}^p
      \ind{\norm{X-Y} \le R}\right] + \sqrt{\Prob \left[ \norm{X-Y} \ge
        R \right] \E \norm{X-Y}^{2q}}.
  \end{align*}
  By Proposition \ref{T:Paouris},
  \[
  \Prob\left[\norm{X-Y} \ge R\right]
  \le \Prob\left[\norm{X} \ge R/2\right] + \Prob\left[\norm{Y} \ge
    R/2\right]
  \le e^{-cR}
  \]
  when $R \ge C \sqrt{n}$, and 
  \[
  \left(\E \norm{X-Y}^{2q} \right)^{1/2q} 
  \le \left(\E \norm{X}^{2q} \right)^{1/2q} + \left(\E \norm{Y}^{2q}\right)^{1/2q}
  \le C \max\{q, \sqrt{n}\},
  \]
  so that
  \[\E \norm{X-Y}^q
  \le R^{q-p} \E \norm{X-Y}^p
  + \left(C \max \{ q, \sqrt{n} \}\right)^q e^{-cR}
  \]
  for every $R \ge C \sqrt{n}$.  Taking the infimum over couplings and then applying Lemma
  \ref{T:min} completes the proof.
\end{proof}

The proof of Proposition \ref{T:h-tv} uses the following variance
bound which follows from a more general concentration inequality due
to Bobkov and Madiman.

\begin{lemma}[{see \cite[Theorem 1.1]{BoMa}}]
  \label{T:BoMa}
  Suppose that $\mu$ is an isotropic log-concave probability measure
  on $\R^n$ with density $f$, and let $Y \sim \mu$. Then
  \[
  \Var \bigl(\log f(Y)\bigr) \le C n.
  \]
\end{lemma}

\begin{proof}[Proof of Proposition \ref{T:h-tv}]
  Let $f$ be the density of $\mu$, and let $\varphi(x) = (2\pi)^{-n/2}
  e^{-\norm{x}^2/2}$ be the density of $\gamma_n$.  Let $Z \sim
  \gamma_n$, $Y \sim \mu$, $X = \frac{f(Z)}{\varphi(Z)}$, and $W =
  \frac{f(Y)}{\varphi(Y)}$.  Then
  \[
  H(\mu \mid \gamma_n) = \E X \log X.
\]
In general, if $\mu$ and $\nu$ have densities $f_\mu$ and $f_\nu$, it
is an easy exercise to show that $d_{TV}(\mu,\nu)=\int
\abs{f_\mu-f_\nu}$; from this, it follows that
\[
  d_{TV}(\mu, \gamma_n) = \E \abs{X-1} = \frac{1}{2} \E (X-1) \ind{X \ge 1}.
  \]

  Let $h(x) = x \log x$.  Since $h$ is convex and $h(1) = 0$, we have
  that $h(x) \le a (x-1)$ for $1 \le x \le R$ as long as $a$ is such that $h(R) \le a
  (R-1)$.  Let $R \ge 2$, so that $\frac{R}{R-1} \le 2$.  Then
  \[
  h(R) = R \log R \le 2 (R-1) \log R = a(R-1)
  \]
  for $a = 2 \log R$.  Thus
  \begin{align*}
  \E X \log X 
  & \le \E (X \log X ) \ind{X \ge 1} \\
  & \le a \E (X-1) \ind{X \ge 1} + \E (X \log X) \ind{X \ge R} \\
  & = (\log R) d_{TV}(\mu, \gamma_n) + \E (X \log X) \ind{X \ge R}.
  \end{align*}

  The Cauchy--Schwarz inequality implies that
  \[
  \E (X \log X) \ind{X \ge R} 
  = \E (\log W) \ind{W \ge R}
  \le \sqrt{\E (\log W)^2} \sqrt{\Prob \left[ W \ge R \right]}.
  \]
  By the $L^2$ triangle inequality, we have
  \begin{align*}
    \sqrt{\E (\log W)^2} &= \sqrt{ \E \abs{\log f(Y) - \log \varphi(Y)}^2} \\
    & \le \sqrt{\E \abs{\log f(Y)}^2}
      + \sqrt{\E \abs{\log \varphi(Y)}^2},
  \end{align*}
  and by Proposition \ref{T:Paouris},
  \begin{align*}
    \E \abs{\log \varphi(Y)}^2
    & = \E \left( \frac{n}{2} \log 2\pi + \frac{\norm{Y}^2}{2}
    \right)^2 \le C n^2.
  \end{align*}
  By Lemma \ref{T:BoMa},
  \[
  \E \abs{\log f(Y)}^2 \le \left( \E \log f(Y)\right)^2 + C n.
  \]
  Recall that the \textbf{entropy} of $\mu$ is
  \[
  - \int f(y) \log f(y) \ dy = - \E \log f(Y) \ge 0,
  \]
  and that $\gamma_n$ is the maximum-entropy distribution with
  identity covariance, so that
  \[
  \left( \E \log f(Y)\right)^2 \le \left( \E \log \varphi(Z) \right)^2
  = \left( n \log \sqrt{2\pi e} \right)^2.
  \]
  Thus
  \[
  \sqrt{\E (\log W)^2}  \le C n.
  \]

  By \cite[Theorem 5.14(e)]{LoVe}, $\norm{f}_\infty \le 2^{8n} n^{n/2}$, and so
  \begin{align*}
  \Prob \left[ W \ge R \right] 
  &= \Prob \left[ \frac{f(Y)}{\varphi(Y)} \ge R \right]
  \le \Prob\left[ e^{\norm{Y}^2/2} \ge (2^{17} \pi n)^{-n/2} R \right] \\
  & = \Prob \left[ \norm{Y} \ge \sqrt{2 \log \left( (2^{17} \pi
        n)^{-n/2} R \right)} \right]
  \end{align*}
  for each $R \ge (2^{17} \pi n)^{n/2}$.  Proposition \ref{T:Paouris}
  now implies that
  \[
  \Prob \left[ W \ge R \right] 
  \le e^{-c \sqrt{\log R - \frac{n}{2} \log(2^{17} \pi n)}}
  \le e^{ - c' \sqrt{\log R}}
  \]
  for $\log R \ge C n \log (n+1)$. 

  Substituting $S = c\sqrt{\log R}$, all together this shows  that
  \[
  H(\mu \mid \gamma_n) \le C \left(S^2 d_{TV}(\mu, \gamma_n) + n
    e^{-S}\right)
  \]
  for every $S \ge c \sqrt{n \log (n+1)}$.  The result follows using Lemma
  \ref{T:min}.
\end{proof}

\section*{Acknowledgements}

This research was partially supported by grants from the Simons
Foundation (\#267058 to E.M.\ and \#264103 to M.M.) and the
U.S. National Science Foundation (DMS-1308725 to E.M.).  This work was
carried out while the authors were visiting the Institut de
Math\'ematiques de Toulouse at the Universit\'e Paul Sabatier; the
authors thank them for their generous hospitality.


\bibliographystyle{plain}
\bibliography{probability-metrics}

\end{document}